\documentclass[reqno, 12pt]{amsart}
\pdfoutput=1
\makeatletter
\let\origsection=\section \def\section{\@ifstar{\origsection*}{\mysection}} 
\def\mysection{\@startsection{section}{1}\z@{.7\linespacing\@plus\linespacing}{.5\linespacing}{\normalfont\scshape\centering\S}}
\makeatother 

\usepackage{amsmath,amssymb,amsthm}
\usepackage{mathrsfs}
\usepackage{mathabx}\changenotsign
\usepackage{dsfont}
 
\usepackage{xcolor}
\usepackage[backref]{hyperref}
\hypersetup{
    colorlinks,
    linkcolor={red!60!black},
    citecolor={green!60!black},
    urlcolor={blue!60!black}
}

\usepackage{graphicx}

\usepackage[open,openlevel=2,atend]{bookmark}

\usepackage[colorinlistoftodos]{todonotes}

\usepackage[abbrev,msc-links,backrefs]{amsrefs} 
\usepackage{doi}

\renewcommand{\PrintDOI}[1]{\doi{#1}}

\usepackage[T1]{fontenc}
\usepackage{lmodern}
\usepackage[babel]{microtype}
\usepackage[english]{babel}

\usepackage{geometry}
\numberwithin{equation}{section}
\numberwithin{figure}{section}

\usepackage{enumitem}

\let\polishlcross=\l
\def\l{\ifmmode\ell\else\polishlcross\fi}

\makeatletter
\def\moverlay{\mathpalette\mov@rlay}
\def\mov@rlay#1#2{\leavevmode\vtop{   \baselineskip\z@skip \lineskiplimit-\maxdimen
   \ialign{\hfil$\m@th#1##$\hfil\cr#2\crcr}}}
\newcommand{\charfusion}[3][\mathord]{
    #1{\ifx#1\mathop\vphantom{#2}\fi
        \mathpalette\mov@rlay{#2\cr#3}
      }
    \ifx#1\mathop\expandafter\displaylimits\fi}
\makeatother

\DeclareFontFamily{U}  {MnSymbolC}{}
\DeclareSymbolFont{MnSyC}         {U}  {MnSymbolC}{m}{n}
\DeclareFontShape{U}{MnSymbolC}{m}{n}{
    <-6>  MnSymbolC5
   <6-7>  MnSymbolC6
   <7-8>  MnSymbolC7
   <8-9>  MnSymbolC8
   <9-10> MnSymbolC9
  <10-12> MnSymbolC10
  <12->   MnSymbolC12}{}
\DeclareMathSymbol{\powerset}{\mathord}{MnSyC}{180}

\theoremstyle{plain}
\newtheorem{theorem}{Theorem}[section]
\newtheorem{thm}[theorem]{Theorem}

\newtheorem{lem}[theorem]{Lemma}
\newtheorem{proposition}[theorem]{Proposition}

\theoremstyle{definition}

\newtheorem{defn}[theorem]{Definition}

\newtheorem{cor}[theorem]{Corollary}

\theoremstyle{remark}
\newtheorem{remark}[theorem]{Remark}

\global\long\def\KERNELSPACE{\mathcal{W}}

\title{Two remarks on graph norms}
\author[F. Garbe]{Frederik Garbe}
\author[J. Hladk\'y]{Jan Hladk\'y}
\address{(FG,JH) Institute of Mathematics of the Czech Academy of Sciences, \v{Z}itn\'a 25, 115~67 Prague, Czechia. Supported by GA\v{C}R project 18-01472Y. With institutional
	support RVO:67985840}
\author[J. Lee]{Joonkyung Lee}
\address{(JL) Fachbereich Mathematik, Universit\"at Hamburg, Germany. Supported by ERC Consolidator Grant PEPCo 724903.}
\begin{document}
\maketitle
\begin{abstract}
    For a graph $H$, its homomorphism density in graphs naturally extends to the space of two-variable symmetric functions $W$ in $L^p$, $p\geq e(H)$, denoted by~$t(H,W)$. One may then define corresponding functionals~$\|W\|_{H}:=|t(H,W)|^{1/e(H)}$ and $\|W\|_{r(H)}:=t(H,|W|)^{1/e(H)}$ and say that $H$ is (semi-)norming if $\|.\|_{H}$ is a (semi-)norm and that $H$ is weakly norming if $\|.\|_{r(H)}$ is a norm.
     
    We obtain two results that contribute to the theory of (weakly) norming graphs. 
    Firstly, answering a question of Hatami, who estimated the modulus of convexity and smoothness of $\|.\|_{H}$, we prove that $\|.\|_{r(H)}$ is not uniformly convex nor uniformly smooth, provided that $H$ is weakly norming.
    Secondly, we prove that every graph $H$ without isolated vertices is (weakly) norming if and only if each component is an isomorphic copy of a (weakly) norming graph.
    This strong factorisation result allows us to assume connectivity of $H$ when studying graph norms.
    In particular, we correct an error in the original statement of the aforementioned theorem by Hatami. 
\end{abstract}
\section{Introduction}
One of the cornerstones of the theory of quasirandomness, due to Chung--Graham--Wilson~\cite{Chung1989} and to Thomason~\cite{Thomason1987}, is that a graph is quasirandom if and only if it admits a random-like count for any even cycle. A modern interpretation of this phenomenon is that the even cycle counts are essentially equivalent to the Schatten--von Neumann norms on the space of two variable symmetric functions, which are the natural limit object of large dense graphs. 
Indeed, Lov\'asz~\cite{LovaszHom} asked the natural question whether other graph counts can also induce a similar norm, which motivated Hatami's pioneering work~\cite{Hat:Siderenko} in the area. Since then, graph norms have been an important concept in the theory of graph limits and received considerable attention. For instance, Conlon and the third author~\cite{ConLeeNorm} obtained a large class of graph norms, Kr\'a\v{l}, Martins, Pach, and Wrochna~\cite{KMPW:StepSidorenko} proved that edge-transitive non-norming graphs exist, and very recently, the first author with Dole\v{z}al, Greb\'ik, Rocha, and Rozho\v{n}~\cite{DGHRR:Parameters} linked graph norms to the so-called step Sidorenko property.

The current note contributes further to this emerging theory of graph norms. We recall the basic definitions given in Hatami's work~\cite{Hat:Siderenko} with slight modifications taken from~\cite{Lovasz2012}. 
Let~$\Omega$ be an arbitrary standard Borel space with an atomless probability measure $\nu$. Whenever we consider a subset of $\Omega$, we tacitly assume that it is measurable. We denote by $\KERNELSPACE$ the linear space of all bounded symmetric measurable functions $W:\Omega^2\rightarrow\mathbb R$. Also let $\KERNELSPACE_{\ge 0}\subseteq\KERNELSPACE$ be the set of non-negative functions in $\KERNELSPACE$. Recall that functions in $\KERNELSPACE_{\ge 0}$ that are bounded above by~1 are called \emph{graphons}, and arise as limits of graph sequences~\cite{LovaszSzegedy}.

Let $H$ be a graph on the vertex set $\{v_1,\ldots,v_n\}$. Given a symmetric measurable real-valued function $W$ on $\Omega^2$, set
\begin{equation}\label{eq:defden}
t(H,W):=\int_{x_{1}\in\Omega}\ldots\int_{x_{n}\in\Omega}\prod_{\{v_{i},v_{j}\}\in E(H)}W(x_{i},x_{j})~d\nu^{\otimes n}\;.
\end{equation}
Let $\mathcal{W}_H$ (resp. $\mathcal{W}_{r(H)}$) be the set of those symmetric measurable functions $W:\Omega^2\rightarrow\mathbb R$ for which $|t(H,W)|$ (resp. $t(H,|W|)$) is defined and is finite. Obviously, $\mathcal{W}_H$ is a subspace of $\mathcal{W}_{r(H)}$, and H\"older's inequality immediately proves that $L^p(\Omega^2)$ is contained in~$\mathcal{W}_{r(H)}$ whenever $p\geq e(H)$.

We then say that $H$ is \emph{(semi-)norming} if $\|\cdot\|_H:=|t(H,\cdot)|^{1/e(H)}$ is a (semi-)norm on~$\mathcal{W}_H$. Likewise, we say that $H$ is \emph{weakly norming} if $\|\cdot\|_{r(H)}:=t(H,|\cdot|)^{1/e(H)}$ is a norm on~$\mathcal{W}_{r(H)}$. Since $\KERNELSPACE$ is a dense subset of the Banach space\footnote{By the topological equivalence between the cut norm and graph norms (see, for instance, Section~5.2 in~\cite{ConLeeNorm}) and compactness of $\mathcal{W}$ under the cut norm, $\|\cdot\|_{r(H)}$ and $\|\cdot\|_H$ also define Banach spaces.} $(\mathcal{W}_H,\|\cdot\|_H)$, this definition does not depend on whether we work in the Banach space $(\KERNELSPACE,\|\cdot\|_H)$ or $(\mathcal{W}_H,\|\cdot\|_H)$. Analogously, in the definition of weakly norming property, $\mathcal{W}_{r(H)}$ can be replaced by $\KERNELSPACE$.

\medskip

In what follows, we shall give short proofs of two results concerning (weakly) norming graphs. Firstly, we study basic geometric properties of the space~$(\mathcal{W}_{r(H)},\|\cdot\|_{r(H)})$.
The definitions of uniform smoothness and uniform convexity will be precisely given in the next section.
\begin{theorem}\label{thm:uniform}
Let $H$ be a weakly norming graph. 
Then the normed space $(\mathcal{W}_{r(H)},\|\cdot\|_{r(H)})$ is not uniformly smooth nor uniformly convex.
\end{theorem}
This answers a question of Hatami, who proved that $(\mathcal{W},\|\cdot\|_{H})$ \emph{is} uniformly smooth and uniformly convex whenever $H$ is seminorming and asked for a counterpart of his theorem for weakly norming graphs.

Theorem~\ref{thm:uniform} not only answers a natural question arising from a functional-analytic perspective, but is also meaningful in the theory of quasirandomness. In~\cite{DGHRR:Parameters}, Hatami's theorem about uniform convexity and smoothness (see Theorem~\ref{thm:HatModuli} for a precise statement) is the key ingredient in proving that every norming graph has the `step forcing property'.
By inspecting the proof in~\cite{DGHRR:Parameters}, one may see that the same conclusion for weakly norming graphs $H$ (except forests) could also be obtained if $\|\cdot\|_{r(H)}$ defined a uniformly convex space. However, Theorem~\ref{thm:uniform} proves that such a modification is impossible.

\medskip

Secondly, we prove a strong `factorisation' result for disconnected weakly norming graphs.
\begin{theorem}\label{thm:factor}
A graph $H$ is weakly norming if and only if all its non-singleton connected components are isomorphic and weakly norming. The same statement with weakly norming replaced by either seminorming or norming also holds.
\end{theorem}
The `if' direction is obvious, since $|t(H,W)|^{1/e(H)}=|t(H',W)|^{1/e(H')}$ whenever $W\in\KERNELSPACE$ and $H$ is a vertex-disjoint union of copies of $H'$ and an arbitrary number of isolated vertices, but the converse is non-trivial. 

Theorem~\ref{thm:factor} corrects a number of errors that assume connectivity of graphs without stating it, which in fact appeared in multiple papers on graph norms including Hatami's work~\cite{Hat:Siderenko}.
We also remark that for Sidorenko's conjecture, a major open problem in extremal combinatorics, even a weak factorisation result --- such as each component of a graph satisfying the conjecture again satisfies it --- is unknown,
even though weakly norming graphs satisfy the conjecture.
In fact, Conlon and the third author~\cite[Corollary~1.3]{ConLee:Sidorenko} proved that the weak factorisation result, if it exists, implies the full conjecture.

\medskip

\section{Moduli of convexity and smoothness}\label{sec:Moduli}
We begin by recalling the definitions of moduli of convexity and moduli of smoothness of a normed space.
\begin{defn}
Let $\left(X,\left\| \cdot\right\| \right)$ be a normed
space. The \emph{modulus of convexity of} $X$ is a function
$\mathfrak{d}_{X}:(0,2]\rightarrow\mathbb{R}$ defined by 
\begin{align}\label{eq:defmodconv}
\mathfrak{d}_{X}(\varepsilon) & :=\inf\left\{ 1-\left\| \frac{x+y}{2}\right\| \::\:x,y\in X,\left\| x-y\right\| \ge\varepsilon,\left\| x\right\| =\left\| y\right\| =1\right\} .
\end{align}
The \emph{modulus of smoothness of} $X$ is a function $\mathfrak{s}_{X}:(0,\infty)\rightarrow\mathbb{R}$
defined by 
\begin{align}\label{eq:defmodsmooth}
\mathfrak{s}_{X}(\varepsilon) & :=\sup\left\{ \frac{1}{2}\left(\left\| x+y\right\| +\left\| x-y\right\| -2\right)\::\:x,y\in X,\left\| x\right\| =1,\left\| y\right\| =\varepsilon\right\} .
\end{align}

The normed space $(X,\|\cdot\|)$ is \emph{uniformly convex}  if $\mathfrak{d}_{X}(\varepsilon)>0$
for each $\varepsilon>0$ and is \emph{uniformly smooth}
if $\lim_{\varepsilon\searrow0}\frac{\mathfrak{s}_{X}(\varepsilon)}{\varepsilon}=0$. For convenience, we write $\mathfrak{d}_{H}$, $\mathfrak{s}_{H}$, $\mathfrak{d}_{r(H)}$ and $\mathfrak{s}_{r(H)}$ instead of $\mathfrak{d}_{\mathcal{W}_H}$, $\mathfrak{s}_{\mathcal{W}_H}$, $\mathfrak{d}_{\mathcal{W}_{r(H)}}$ and $\mathfrak{s}_{\mathcal{W}_{r(H)}}$, respectively. 
\end{defn}

Hatami~\cite{Hat:Siderenko} determined $\mathfrak{d}_{H}$ and $\mathfrak{s}_{H}$ for connected norming graphs $H$ up to a multiplicative constant by relating them to the moduli of convexity and of smoothness of $\ell^p$-spaces, which are well-understood.
\begin{thm}[Theorem 2.16 in~\cite{Hat:Siderenko}]\label{thm:HatModuli}
	For each $m\in \mathbb{N}$, there exist constants $C_m,C'_m>0$ such that the following holds: let $H$ be a connected seminorming graph with $m$ edges. Then the Banach space $(\mathcal{W}_H,\|\cdot\|_H)$ satisfies $C_m\cdot \mathfrak{d}_{\ell^m}\le \mathfrak{d}_{H}\le \mathfrak{d}_{\ell^m}$ and $\mathfrak{s}_{\ell^m}\le \mathfrak{s}_{H}\le C'_m\cdot \mathfrak{s}_{\ell^m}$.
\end{thm} 
Since for each $p\in(1,+\infty)$, it is well-known that the $\ell^p$-space is uniformly convex and uniformly smooth, one obtains the following.
\begin{cor}\label{cor:HatModuli}
	Let $H$ be a connected seminorming graph. Then the Banach space $(\mathcal{W}_H,\|\cdot\|_H)$ is uniformly convex and uniformly smooth.
\end{cor}

The connectivity of $H$ in Theorem~\ref{thm:HatModuli} was in fact neglected in the original statement in~\cite{Hat:Siderenko}, but it is certainly necessary. For example, by taking a disjoint union of two isomorphic norming graphs with $m/2$ edges (assume $m$ is even), one obtains another norming graph with $m$ edges that gives exactly the same norm, whose correct parameters in Theorem~\ref{thm:HatModuli} are $\mathfrak{d}_H =\Theta(\mathfrak{d}_{\ell^{m/2}})$ and $\mathfrak{s}_H =\Theta(\mathfrak{s}_{\ell^{m/2}})$.
Indeed, in Theorem~\ref{thm:ModuliCorrected} below we obtain a general statement without assuming connectivity by using Theorem~\ref{thm:factor}. But first, let us point out the subtle error in~\cite{Hatami2012} causes that the proof of Theorem~\ref{thm:HatModuli} does not work for disconnected graphs. This error lies in proving $\mathfrak{d}_{H}\leq \mathfrak{d}_{\ell^m}$ and $\mathfrak{s}_{\ell^m}\le \mathfrak{s}_{H}$ by claiming that the Banach space $(\mathcal{W}_H,\|\cdot\|_H)$ contains a subspace isomorphic to $(\ell^m,\|\cdot\|_m)$.
Here we give a full proof of the claim, which in turn reveals where the connectivity of $H$ is used.
To this end, we introduce the following notation, which will also be useful in Section~\ref{sec:disconnected}.
\begin{defn}\label{def:specialgraphon}
Let $\Omega$ be partitioned as $\Omega=\Omega_1\sqcup\Omega_2\sqcup\ldots$ with countably many parts such that $\nu(\Omega_i)=2^{-i}$ for every $i\in\mathbb N$. For each $m\in\mathbb{N}$, $\gamma>0$, and $\mathbf{a}=(a_1,a_2,\ldots)\in \ell^m$,
$W_{\gamma,\mathbf{a}}$~denotes the function satisfying $W_{\gamma,\mathbf{a}}(x,y)=2^{i\gamma}\cdot a_i$ whenever $(x,y)\in\Omega_i^2$ and $W_{\gamma,\mathbf{a}}=0$ outside $\bigcup_i \Omega_i^2$. 
\end{defn}

Suppose that $H$ is a norming graph with $n$ vertices and $m$ edges. In particular this implies that $m$ is even (see \cite[Exercise~14.8]{Lovasz2012}). The map $\mathbf{a}\mapsto W_{\frac{n}{m},\mathbf{a}}$ is linear, and thus, proving that this map preserves the respective norms is enough to conclude that the subspace spanned by $W_{\frac{n}{m},\mathbf{a}}$ is isomorphic to $\ell^m$. 
For each $\mathbf{a}=(a_1,a_2,\ldots)\in \ell^m$,
$$\|\mathbf{a}\|^m_m=\sum_i a_i^m=\sum_i \frac{1}{2^{in}}\cdot(2^{in/m}\cdot a_i)^m=t\big(H,W_{\frac{n}{m},\mathbf{a}}\big)\;.$$
Indeed, if $x_1,\ldots,x_n$ do not fall into any single $\Omega_i$, connectedness of $H$ implies that the product in~\eqref{eq:defden} evaluates to~0. Otherwise, if $(x_1,\ldots,x_n)\in \Omega_i^n$ for some $i\in\mathbb{N}$, then $\nu^{\otimes n}(\Omega_i^n)=\frac{1}{2^{in}}$ and the product in~\eqref{eq:defden} evaluates to constant~$(2^{in/m}\cdot a_i)^m$, 
which proves the last equality. This is exactly where the proof of the claim relies on $H$ being connected.

\medskip
Now, turning to weakly norming graphs, Theorem~\ref{thm:uniform} is a direct consequence of the following result.
\begin{thm}\label{thm:rHnonnconvexnonsmooth}
	Let $H$ be a weakly norming graph. Then for each $\varepsilon\in(0,1)$,
	\begin{enumerate}[label=(\alph*)]
\item\label{en:convexity} $\mathfrak{d}_{r(H)}(\varepsilon)=0$, and
\item\label{en:smoothness} $\mathfrak{s}_{r(H)}(\varepsilon)\ge\frac{1}{2}\varepsilon$.
	\end{enumerate}
\end{thm}
For the proof, we introduce a random graphon model that generalises graphon representations of the Erd\H{o}s--R\'enyi random graph. 
Let $\mathcal{D}$ be a probability distribution on $[0,1]$ and let $\Omega=\Omega_1\sqcup\ldots\sqcup\Omega_n$ be an arbitrary partition of $\Omega$ into sets of measure $\frac{1}{n}$. 
Denote by~$\mathbb{U}(n,\mathcal{D})$ the random graphon obtained by assigning a constant value generated independently at random by the distribution $\mathcal{D}$
on each $(\Omega_i\times \Omega_j) \cup (\Omega_j\times \Omega_i)$, $1\le i\le j\le n$. 
Although $\mathbb{U}(n,\mathcal{D})$ depends on the partition $\Omega_1\sqcup\ldots\sqcup\Omega_n$,
we shall suppress the dependency parameter as
different $\mathbb{U}(n,\mathcal{D})$'s are `isomorphic' in the sense that there exists a measure-preserving bijection that maps one partition to the other. We use the term \emph{asymptotically almost surely}, or \emph{a.a.s.} for short, in the standard way, i.e., a property $\mathcal{P}$ of $\mathbb{U}(n,\mathcal{D})$ holds a.a.s. if the probability that $\mathcal{P}$ occurs tends to $1$ as $n\rightarrow\infty$. We write $a=b\pm \epsilon$ if and only if $a\in [b-\epsilon,b+\epsilon]$.
\begin{proposition}\label{prop:densitiesU}
Let $\mathcal{D}$ be a probability distribution on $[0,1]$ and let $d=\mathbb{E}[\mathcal{D}]$. Then for any fixed graph $H$,  ~$U\sim\mathbb{U}(n,\mathcal{D})$ satisfies $t(H,U)=d^{e(H)}\pm o_n(1)$ a.a.s.
\end{proposition}
We omit the proof, as it is a straightforward application of the standard concentration inequalities to subgraph densities in Erd\H{o}s--R\'enyi random graphs (see, for example, \cite[Corollary~10.4]{Lovasz2012}).

\begin{proof}[Proof of Theorem~\ref{thm:rHnonnconvexnonsmooth}]
Throughout the proof, we briefly write $\|\cdot\|_{r(H)}=\|\cdot\|$.
 For $x\in[0,1]$, denote by $\mathbf{1}\{x\}$ the Dirac measure on $x$. Set
 \begin{align*}
     \mathcal{D}_1&:=\tfrac12 \cdot\mathbf{1}\{0\}+\tfrac12 \cdot\mathbf{1}\{1\}.
 \end{align*}

 Let $U_1$ and $U_2$ be two independent copies of $\mathbb{U}(n,\mathcal{D}_1)$. 
 Proposition~\ref{prop:densitiesU} then implies a.a.s. 
\begin{equation}\label{eq:firstruss}
\left\| U_{i}\right\|=t(H,U_i)^{1/e(H)}=\tfrac{1}{2}\pm o_{n}(1), \quad\mbox{for $i=1,2$.}
\end{equation}
For each $i=1,2$, let $U^*_i:=\frac{1}{2\left\| U_{i}\right\|}\cdot U_i$ be the normalisation of $U_i$ that satisfies $\| U^*_{i}\|=\frac{1}{2}$.
Then the triangle inequality together with~\eqref{eq:firstruss} implies
\begin{equation}\label{eq:normdiff}
\|U^*_i-U_i\|\leq\big|\|U^*_i\|-\|U_i\|\big|=o_n(1)\;.
\end{equation}

Since the random graphon $\left|U_{1}-U_{2}\right|$ is also distributed like $\mathbb{U}(n,\mathcal{D}_1)$, we again have $
	\left\| U_{1}-U_{2}\right\|=\frac{1}{2}\pm o_{n}(1)
$ a.a.s. Thus, by the triangle inequality and~\eqref{eq:normdiff}, $2U^*_{1}$ and	$2U^*_{2}$ are two symmetric functions with $\| 2U^*_{1}\|=\| 2U^*_2\|=1$ whose linear combination is always close to the corresponding one of $U_1$ and $U_2$, i.e., for any fixed $\alpha,\beta\in\mathbb{R}$,
\begin{align}\label{eq:linearcombi}
    \big|\left\| \alpha U_1+\beta U_2 \right\|-\left\| \alpha U^*_1+\beta U^*_2 \right\|\big|
    \leq|\alpha|\left\| U_1-U^*_1\right\|+|\beta|\left\|U_2-U^*_2 \right\|= o_n(1).
\end{align}
In particular, $\alpha=2$ and $\beta=-2$ give 
	$
	\left\| 2U^*_{1}-2U^*_{2}\right\|\ge \left\| 2U_{1}-2U_{2}\right\|-o_n(1)=1\pm o_{n}(1).
	$
That is, $2U_1^*$ and $2U_2^*$ are points on the unit sphere that are `far' apart. 
Setting $\alpha=\beta=1$ in~\eqref{eq:linearcombi} gives
    $\big|\left\| U_1+U_2 \right\|-\left\| U^*_1+U^*_2 \right\|\big|
    = o_n(1)$, and therefore, for any $0<\varepsilon<1$,
	\begin{align}\label{eq:modconvex}
	    \mathfrak{d}_{r(H)}(\varepsilon)\le 1-\left\| \frac{2U^*_1+2U^*_2}{2}\right\|=1-\left\| \frac{2U_1+2U_2}{2}\right\|\pm o_n(1).
	\end{align}
Now let
\begin{align*}
 \mathcal{D}_2&:=\tfrac14 \cdot\mathbf{1}\{0\}+\tfrac12 \cdot\mathbf{1}\{\tfrac12\}+\tfrac14 \cdot\mathbf{1}\{1\}.
\end{align*}
Then, since $\frac12(U_{1}+U_{2})$ has distribution~$\mathbb{U}(n,\mathcal{D}_2)$ and $\mathbb{E}[\mathcal{D}_2]=\frac12$, we have by Proposition~\ref{prop:densitiesU} a.a.s. $\left\| U_{1}+U_{2}\right\|=1\pm o_{n}(1)$.
Substituting this into~\eqref{eq:modconvex} proves that the modulus of convexity of $\|\cdot\|$ is~0 for each $\varepsilon\in(0,1)$.

\medskip

For $\varepsilon\in(0,1)$ given in~\ref{en:smoothness}, let
\begin{align*}
    \mathcal{D}_3:=\tfrac14 \big(\mathbf{1}\{0\}+
    \mathbf{1}\{\varepsilon\}+\mathbf{1}\{1-\varepsilon\}+\mathbf{1}\{1\}\big)~\text{ and }~
\mathcal{D}_4:=\tfrac14 \big(\mathbf{1}\{0\}+\mathbf{1}\{\tfrac{\varepsilon}{2}\}+\mathbf{1}\{\tfrac12\}+\mathbf{1}\{\tfrac{1+\varepsilon}{2}\}\big).
\end{align*}
The distributions of $\left|U_{1}-\varepsilon U_{2}\right|$ and $\frac12\left|U_{1}+\varepsilon U_{2}\right|$ are $\mathbb{U}(n,\mathcal{D}_3)$ and  $\mathbb{U}(n,\mathcal{D}_4)$, respectively.
As $\mathbb{E}[\mathcal{D}_3]=\frac12$ and $\mathbb{E}[\mathcal{D}_4]=\frac{1+\varepsilon}4$, Proposition~\ref{prop:densitiesU} yields that, a.a.s., $\|2U_{1}-2\varepsilon U_{2}\|=1\pm o_n(1)$ and $\|2U_{1}+2\varepsilon U_{2}\|=1+\varepsilon\pm o_n(1)$. 
Therefore, by~\eqref{eq:linearcombi}, $\|2U^*_{1}-2\varepsilon U^*_{2}\|=1\pm o_n(1)$ and $\|2U^*_{1}+2\varepsilon U^*_{2}\|=1+\varepsilon\pm o_n(1)$ a.a.s. Hence, substituting $2U^*_{1}$ and $2\varepsilon U^*_{2}$ into~\eqref{eq:defmodsmooth} gives
$$\mathfrak{s}_{X}(\varepsilon)\ge \frac{1}{2}\left(\left\| 2U^*_{1}+2\varepsilon U^*_{2}\right\| +\left\| 2U^*_{1}-2\varepsilon U^*_{2}\right\| -2\right)=\frac{\varepsilon}{2}\pm o_n(1),$$
which proves~\ref{en:smoothness}.
\end{proof}

\medskip

\section{Disconnected (semi-)norming and weakly norming graphs}\label{sec:disconnected}
To be precise, we expand Theorem~\ref{thm:factor} to two parallel statements, also omitting any isolated vertices from $H$ (this operation does not change $t(H,\cdot)^{1/e(H)}$).
\begingroup
\def\thetheorem{\ref{thm:factor}}
\begin{theorem}[Restated]
For a graph $H$ without isolated vertices, the following holds:
\begin{enumerate}[label=(\alph*)]
\item A graph $H$ is weakly norming if and only if all connected components of $H$ are isomorphic and weakly norming.
\item A graph $H$ is (semi-)norming if and only if all connected component of $H$ are isomorphic and (semi-)norming.
\end{enumerate}
\end{theorem}
\addtocounter{theorem}{-1}
\endgroup

To prove this theorem, we need some basic facts about weakly norming graphs. 
Given a graph $H$ and a collection $\mathbf{w}=(W_e)_{e\in E(H)}\in\KERNELSPACE^{E(H)}$, define 
the \emph{$\mathbf{w}$-decorated homomorphism density} by
$$t(H,\mathbf{w}):=\int_{x_{1}\in\Omega}\ldots\int_{x_{n}\in\Omega}\prod_{e=ij\in E(H)}W_{e}(x_{i},x_{j})\;.$$
That is, we assign a possibly different $W_e$ to each $e\in E(H)$ and are counting such `multicoloured' copies of $H$.
In particular, if $W_e=W$ for all $e\in E(H)$, we obtain $t(H,\mathbf{w})=t(H,W)$.
Hatami~\cite{Hat:Siderenko} observed that the (weakly) norming property is equivalent to a H\"older-type inequality for the decorated homomorphism density.

\begin{lem}[Theorem 2.8 in~\cite{Hat:Siderenko}]\label{lem:Holder}
Let $H$ be a graph. Then
\begin{enumerate}[label=(\alph*)]
    \item $H$ is weakly norming if and only if, for every $w\in \KERNELSPACE_{\ge 0}^{E(H)}$, $$t(H,\mathbf{w})^{e(H)}\le \prod_{e\in E(H)} t(H,W_e).$$ 
    \item $H$ is seminorming if and only if, for every $w\in \KERNELSPACE^{E(H)}$, $$t(H,\mathbf{w})^{e(H)}\le \prod_{e\in E(H)} \left|t(H,W_e)\right|.$$ 
\end{enumerate}
\end{lem}
As the second inequality is more general than the first, it immediately follows that every seminorming graph is weakly norming.
Another easy consequence of this charaterisation is that, for a weakly norming graph $H$, its subgraph $F$, and $W\in\KERNELSPACE_{\geq 0}$, we have the inequality
\begin{align}\label{eq:dominate}
   t(F,W)\leq t(H,W)^{e(F)/e(H)}.
\end{align}
Indeed, one can easily prove this by setting $W_e=W$ for $e\in E(F)$ and $W_e\equiv1$ otherwise.
For yet another application, we use Lemma~\ref{lem:Holder} to prove that a weakly norming graph essentially has no subgraph with larger average degree.
\begin{lem}\label{lem:weaklynormingAvgDeg}
Let $H$ be a weakly norming graph without isolated vertices and let $F$ be its subgraph. Then $\frac{e(F)}{v(F)}\leq \frac{e(H)}{v(H)}$.
\end{lem}
\begin{proof}
We may assume $F$ has no isolated vertices either, as adding isolated vertices only reduces the average degree.
Let $X\subseteq\Omega$ be a subset with $\nu(X)=1/2$ and let $U:\Omega^2\rightarrow [0,1]$ be the graphon defined by $W(x,y)=1$ if $x,y\in X$ and $0$ otherwise.
Then, for any graph $J$ without isolated vertices, $t(J,U)=2^{-v(J)}$.
Choosing $W_e=U$ for $e\in E(F)$ and $W_e\equiv1$ otherwise for $\mathbf{w}\in \KERNELSPACE_{\geq 0}^{E(H)}$ then gives
\begin{align*}
    t(F,U)^{e(H)}=t(H,\mathbf{w})^{e(H)}\leq t(H,U)^{e(F)}t(H,1)^{e(H)-e(F)}=t(H,U)^{e(F)}.
\end{align*}
Comparing $t(F,U)^{e(H)}=2^{-v(F)e(H)}$ and $t(H,U)^{e(F)}=2^{-v(H)e(F)}$ concludes the proof.
\end{proof}
\begin{remark}
This is reminiscent of~\cite[Theorem~2.10(i)]{Hat:Siderenko}, which states that $\frac{e(F)}{v(F)-1}\leq \frac{e(H)}{v(H)-1}$ whenever $H$ is weakly norming and $F$ is a subgraph of~$H$ with $v(F)>1$. However, this theorem is only true if $H$ is connected and hence also needs to be corrected. To see this, let $H$ be a vertex disjoint union of two copies of $K_{1,2}$, which is a norming graph. Then $\frac{e(H)}{v(H)-1}=4/5$ but $\frac{e(F)}{v(F)-1}=1$ for $F=K_{1,2}$.
\end{remark}

Suppose now that a weakly norming graph $H$ without isolated vertices consists of two vertex-disjoint subgraphs $F_1$ and $F_2$. If $e(F_1)/v(F_1)>e(F_2)/v(F_2)$, then
\begin{align*}
    \frac{e(H)}{v(H)} = \frac{e(F_1)+e(F_2)}{v(F_1)+v(F_2)}
    <\frac{e(F_1)}{v(F_1)},
\end{align*}
which contradicts to Lemma~\ref{lem:weaklynormingAvgDeg}. By iterating this, we obtain the following fact.
\begin{cor}\label{cor:avgdeg}
    Every component in a weakly norming graph without isolated vertices has the same average degree.
\end{cor}
Before proceeding to the next step, we recall some basic facts about $\ell^p$-spaces. For $0<p<q\le +\infty$ we have $\|\cdot \|_p \ge \|\cdot \|_q$. Furthermore, there exists $\mathbf{c}\in\ell^\infty$ such that
\begin{equation}\label{eq:pnormqnorm}
    \| \mathbf{c} \|_p > \| \mathbf{c} \|_q\;.
\end{equation}

\begin{lem}\label{lem:weaklynormingEdgesEqual}
In a weakly norming graph $H$ without isolated vertices, every connected component has the same number of edges.
\end{lem}
\begin{proof}
Let $F_1,\ldots,F_k$ be the connected components of $H$ and let  $\gamma:=\frac{v(F_1)}{e(F_1)}$. By Corollary~\ref{cor:avgdeg},  $\frac2{\gamma}$ is the average degree of all $F_i$, $i=1,2,\ldots,k$. Recall the definition of $W_{\gamma,\mathbf{a}}$ given in Definition~\ref{def:specialgraphon}. For each $\mathbf{a}=(a_1,a_2,\ldots)\in\ell^\infty$ and each connected graph $F$ which also has average degree $\frac2{\gamma}$, and, say, $m$ edges, we have 
\begin{align}
\begin{split}\label{eq:tFalpha}
    t(F,|W_{\gamma,\mathbf{a}}|)&=\sum_i |a_i|^{m} 
    =\|\mathbf{a} \|_{m}^{m}\;.
\end{split}
\end{align}

Suppose that not all the components have the same number of edges. Let $p=\min_j e(F_j)$. We may assume that $p=e(F_1)$. Let $q>p$ be the number of edges in a component with more edges than $F_1$ and let $\mathbf{c}\in \ell^\infty$ be given by~\eqref{eq:pnormqnorm}. 
Define the collection $\mathbf{w}=(W_e)_{e\in E(H)}$ by $W_e=|W_{\gamma,\mathbf{c}}|$ for $e\in E(F_1)$ and $W_e\equiv1$ otherwise. Lemma~\ref{lem:Holder} then gives
\begin{equation}\label{eq:IlswA}
t(F_1,|W_{\gamma,\mathbf{c}}|)^{e(H)}=t(H,\mathbf{w})^{e(H)}\le \prod_{e\in E(H)} t(H,W_e)=t(H,|W_{\gamma,\mathbf{c}}|)^{p}.
\end{equation}
Expanding the term $t(H,|W_{\gamma,\mathbf{c}}|)$ on the right-hand side of~\eqref{eq:IlswA} using~\eqref{eq:tFalpha} yields 
$$t(H,|W_{\gamma,\mathbf{c}}|)=\prod_{j=1}^k t(F_j,|W_{\gamma,\mathbf{c}}|)=\prod_{j=1}^k \|\mathbf{c}\|_{e(F_j)}^{e(F_j)}\;.$$
On the left-hand side of~\eqref{eq:IlswA}, we have by~\eqref{eq:tFalpha} that $t(F_1,|W_{\gamma,\mathbf{c}}|)=\|\mathbf{c}\|_{p}^{p}$. 
Substituting these back to~\eqref{eq:IlswA} gives
$$\|\mathbf{c}\|_{p}^{p\cdot e(H)}\le \left(\prod_{j=1}^k \|\mathbf{c}\|_{e(F_j)}^{e(F_j)}\right)^p,$$
which contradicts to the fact that $\|\mathbf{c}\|_{p}\ge \|\mathbf{c}\|_{e(F_j)}$ for each $j\in[k]$ with at least one of the inequalities being strict.
\end{proof}
\begin{lem}\label{lem:weaklynormingIsomorphic}
For a weakly norming graph $H$ without isolated vertices, all the components of $H$ are isomorphic.
\end{lem}
\begin{proof}
Suppose that there are at least two non-isomorphic graphs amongst all the components $F_1,\ldots,F_k$. By Lemma~\ref{lem:weaklynormingEdgesEqual} we may assume that all $F_i$ have the same number of edges, say $m$. In particular, $e(H)=mk$. By Theorem~5.29 in~\cite{Lovasz2012}, there exists a graphon $U$ so that the numbers $t(F_1,U),\ldots,t(F_k,U)$ are not all equal. 
We may assume that $t(F_1,U)$ attains the maximum amongst $t(F_1,U),\ldots,t(F_k,U)$. Then we have
$t(H,U)=\prod_{i=1}^k t(F_i,U)<t(F_1,U)^k$, which contradicts
\begin{align*}
    t(F_1,U)\leq t(H,U)^{m/e(H)}= t(H,U)^{1/k}
\end{align*}
which follows from~\eqref{eq:dominate}.
\end{proof}
\begin{proof}[Proof of Theorem~\ref{thm:factor}]
Suppose first that $H$ is weakly norming.
Let $F$ be the graph given by Lemma~\ref{lem:weaklynormingIsomorphic} which is isomorphic to every component of $H$ and let $k$ be the number of components of $H$.
Now enumerate the edges in $H$ by $(e,i)\in E(F)\times [k]$, where each $(e,i)$ denotes the edge $e$ in the $i$-th copy of $F$.
 Then each $\mathbf{w}\in \KERNELSPACE^{E(H)}$ can be written as $(\mathbf{w}_1,\mathbf{w}_2,\ldots,\mathbf{w}_k)$,
 where $\mathbf{w}_i = (W_{e,i})_{e\in E(F)}$ such that
 $
     t(H,\mathbf{w})=\prod_{i=1}^k t(F,\mathbf{w}_i).
 $
 Let $\mathbf{u}=(U_e)_{e\in E(F)}\in \mathcal{W}_{\geq 0}^{E(F)}$ be arbitrary.
 Then Lemma~\ref{lem:Holder} together with the choice $\mathbf{w}_1=\mathbf{w}_2=\ldots=\mathbf{w}_k=\mathbf{u}$, i.e., $W_{e,i}=U_e$, implies
 \begin{align}\label{eq:HolderProof}
     \nonumber t(F,\mathbf{u})^{k^2\cdot e(F)}&=t(F,\mathbf{u})^{k\cdot e(H)}=t(H,\mathbf{w})^{e(H)}\\&\leq 
     \prod_{(e,i)\in E(H)}t(H,W_{e,i})
     =\prod_{(e,i)\in E(F)\times [k]}t(F,U_{e})^k
     =\prod_{e\in E(F)}t(F,U_{e})^{k^2}.
 \end{align}
Taking the $k^2$-th root proves that $F$ is weakly norming.
 
 When $H$ is seminorming, we can still apply Lemma~\ref{lem:weaklynormingIsomorphic} to obtain a graph $F$ isomorphic to each component, since $H$ is also weakly norming.
 Thus, the enumeration $E(F)\times [k]$ of $E(H)$ and the factorisation $t(H,\mathbf{w})=\prod_{i=1}^k t(F,\mathbf{w}_i)$ for each $\mathbf{w}=(\mathbf{w}_1,\mathbf{w}_2,\ldots,\mathbf{w}_k)\in \KERNELSPACE^{E(H)}$ remain the same.
 Now let $\mathbf{u}=(U_f)_{f\in E(F)}\in \mathcal{W}^{E(F)}$ be arbitrary.
 Then again by taking $\mathbf{w}_1=\mathbf{w}_2=\ldots=\mathbf{w}_k=\mathbf{u}$ in Lemma~\ref{lem:Holder}, we obtain
 \begin{align*}
     t(F,\mathbf{u})^{k^2\cdot e(F)}=t(H,\mathbf{w})^{e(H)}\leq 
     \prod_{(e,i)\in E(H)}|t(H,W_{e,i})|
     =\prod_{e\in E(F)}|t(F,U_{e})|^{k^2},
 \end{align*}
 which proves that $H$ is seminorming.
 If $H$ is norming, then $|t(F,W)|=|t(H,W)|^{1/k}$ must be nonzero for each nonzero $W\in\KERNELSPACE$. Thus,~$F$ is also norming.
\end{proof}

\section{Concluding remarks}
As mentioned in Section~\ref{sec:Moduli}, Theorem~\ref{thm:factor} yields a full generalisation of Theorem~\ref{thm:HatModuli}.
\begin{thm}\label{thm:ModuliCorrected}
	For each $m\in \mathbb{N}$, there exist constants $C_m,C'_m>0$ such that the following holds: let $H$ be a seminorming graph with $m$ edges in each (isomorphic) non-singleton component. Then the Banach space $(\mathcal{W}_H,\|\cdot\|_H)$ satisfies $C_m\cdot \mathfrak{d}_{\ell^m}\le \mathfrak{d}_{H}\le \mathfrak{d}_{\ell^m}$ and $\mathfrak{s}_{\ell^m}\le \mathfrak{s}_{H}\le C'_m\cdot \mathfrak{s}_{\ell^m}$.
\end{thm} 
As a consequence, the connectivity condition in Corollary~\ref{cor:HatModuli} can also be removed, i.e., $(\mathcal{W}_H,\|\cdot\|_H)$ is always uniformly convex and uniformly smooth whenever $H$ is seminorming.

There is more literature in the area that has been imprecise when it comes to connectivity, but which can be corrected by Theorem~\ref{thm:factor} to hold in full generality. For instance, Exercise~14.7(b) in~\cite{Lovasz2012} states that every seminorming graph is either a star or Eulerian, which is true only if the seminorming graph is connected. To correct the statement, we may replace a star by a vertex disjoint union of isomorphic stars by using Theorem~\ref{thm:factor}.
Likewise, whenever studying properties of graph norms, one can invoke Theorem~\ref{thm:factor} and focus on connected graphs.
We finally remark that the theorems used in our proofs have no errors concerning connectivity. In particular, Theorem~2.8 in~\cite{Hat:Siderenko} is still valid regardless of connectivity.

\medskip

In~\cite{KMPW:StepSidorenko}, the \emph{step Sidorenko property} is defined to prove that there exists an edge-transitive graph that is not weakly norming (for the precise definition, we refer to~\cite{KMPW:StepSidorenko}), where the proof relies on the fact from~\cite{Lovasz2012} that every weakly norming graph is step Sidorenko.
Moreover, it is shown in~\cite{DGHRR:Parameters} that the converse is also true for connected graphs, i.e., every connected step Sidorenko graph is weakly norming. However, Theorem~\ref{thm:factor} proves that the converse no longer holds for disconnected graphs, as a vertex-disjoint union of non-isomorphic step Sidorenko graphs is again step Sidorenko but not weakly norming. 

\medskip

\noindent
\textbf{Acknowledgements.} Part of this work was carried out while the third author visited the other authors in Prague and he is grateful for their support and hospitality.

\bibliographystyle{plain}
\bibliography{bibl}	 
\end{document}